\documentclass[a4paper,10pt]{amsart}

\usepackage[english]{babel}
\usepackage[utf8]{inputenc}

\usepackage{mathrsfs}
\usepackage{mathtools}
\usepackage{dsfont}
\usepackage{amssymb}
\usepackage{amsmath}
\usepackage{enumerate}
\usepackage{amsfonts}
\usepackage{amsthm,amsmath}
\usepackage{bbm}
\usepackage{color}
\usepackage{hyperref}

\numberwithin{equation}{section}
\newcommand{\C}{\mathds{C}}
\newcommand{\N}{\mathds{N}}
\newcommand{\Q}{\mathds{Q}}
\newcommand{\R}{\mathds{R}}

\newcommand{\Z}{\mathds{Z}}

\newcommand{\cK}{\mathscr{K}}
\newcommand{\cL}{\mathscr{L}}

\newcommand{\cT}{\mathscr{T}}
\newcommand{\cS}{\mathscr{S}}

\newcommand{\calL}{\mathscr{L}}

\newcommand{\applied}[2]{\langle #1,#2\rangle}
\DeclarePairedDelimiter\norm{\lVert}{\rVert}
\DeclarePairedDelimiter\abs{\lvert}{\rvert}
\newcommand{\argument}{\,\cdot\,}
\renewcommand{\phi}{\varphi}
\newcommand{\colonequiv}{\mathrel{\mathop:}\Leftrightarrow}

\DeclareMathOperator{\lin}{span}
\newcommand{\dx}{\:\mathrm{d}}
\newcommand{\eps}{\varepsilon}

\DeclareMathOperator{\id}{id}

\definecolor{mygreen}{rgb}{0.1,0.75,0.2}

\theoremstyle{definition}
\newtheorem{definition}{Definition}[section]
\newtheorem{remark}[definition]{Remark}

\newtheorem{example}[definition]{Example}

\theoremstyle{plain}
\newtheorem{proposition}[definition]{Proposition}
\newtheorem{lemma}[definition]{Lemma}
\newtheorem{theorem}[definition]{Theorem}
\newtheorem{corollary}[definition]{Corollary}
\newtheorem*{theorem_no_number}{Theorem}

\begin{document}

\title{Convergence of positive operator semigroups}
\author{Moritz Gerlach}
\address{Moritz Gerlach\\Universit\"at Potsdam\\Institut f\"ur Mathematik\\Karl-Liebknecht-Stra{\ss}e 24--25\\14476 Potsdam\\Germany}
\email{moritz.gerlach@uni-potsdam.de}
\author{Jochen Gl\"uck}
\address{Jochen Gl\"uck\\Universit\"at Ulm\\Institut f\"ur Angewandte Analysis\\89069 Ulm, Germany}
\email{jochen.glueck@alumni.uni-ulm.de}

\date{\today}
\begin{abstract}
	We present new conditions for semigroups of positive operators to converge strongly as time tends to infinity.
	Our proofs are based on a novel approach combining the well-known splitting theorem by Jacobs, de Leeuw and Glicksberg 
	with a purely algebraic result about positive group representations. 
	Thus we obtain convergence theorems not only for one-parameter semigroups but for a much larger class of semigroup representations.
	
	Our results allow for a unified treatment of various theorems from the literature 
	that, under technical assumptions, a bounded positive $C_0$-semigroup containing or dominating a kernel operator
	converges strongly as $t \to \infty$. We gain new insights into the structure theoretical background of those theorems and generalise 
	them in several respects; especially we drop any kind of continuity or regularity assumption with respect to the time parameter.
\end{abstract}

\maketitle

\section{Introduction} \label{section:introduction}

One of the most important aspects in the study of linear autonomous evolution equations is the long-term behaviour of their solutions. 
Since these solutions are usually described by one-parameter operator semigroups, it is essential to have affective tools for the analysis 
of their asymptotic behaviour available. 
In many applications the underlying Banach space is a function space and thus exhibits some kind of order structure. 
If, in such a situation, a positive initial value of the evolution equation leads to a positive solution, one speaks of a \emph{positive} operator semigroup. Various approaches were developed to prove, under appropriate assumptions, convergence of such semigroups as time tends to infinity; see the end of the introduction for a very brief overview.

\subsection*{Semigroups of kernel operators} Consider a one-parameter semigroup $(T_t)_{t \in (0,\infty)}$ (meaning that $T_{t+s} = T_tT_s$ for all $t,s > 0$) 
of positive operators on a Banach lattice $E$. In concrete applications it happens frequently that the operators $T_t$ are \emph{kernel operators}
(see Appendix~\ref{appendix:AM-compact} for a precise definition of this notion). Under mild additional assumptions this is already sufficient for $(T_t)_{t\in (0,\infty)}$ to converge strongly as $t \to \infty$, as the following classical result of Greiner \cite[Kor 3.11]{greiner1982} illustrates:

\begin{theorem_no_number}
	Let $\cT = (T_t)_{t \in [0,\infty)}$ be a positive and contractive $C_0$-semigroup on an $L^p$-space over a $\sigma$-finite measure space where $1 \le p < \infty$. 
	Assume that $\cT$ possesses a fixed vector $h$ such that $h > 0$ almost everywhere and that $T_{t_0}$ is a kernel operator for some $t_0 \ge 0$.
	
	Then $(T_t)_{t\in [0,\infty)}$ converges strongly as $t \to \infty$.
\end{theorem_no_number}

It was observed only recently that this result can also be used to give an analytic proof of a famous theorem by Doob about the asymptotic behaviour of one-parameter Markov semigroups, 
see \cite{gerlach2012} and \cite[Sec~4]{gerlach2015}. During the last two decades several new versions of the above theorem were discovered. 
If one assumes the semigroup $\cT$ to be irreducible, then it suffices that $T_{t_0}$ \emph{dominates} a non-trivial kernel operator instead of being a kernel operator itself. 
For one-parameter Markov semigroups on $L^1$-spaces this was observed by Pich\'{o}r and Rudnicki in \cite[Thm~1]{pichor2000}. This result was applied to study the asymptotic behaviour of numerous models, 
in particular from mathematical biology; several examples can be found, for instance, in \cite{rudnicki2003, bobrowski2007, mackey2008, du2011, banasiak2012a, mackey2013} 
and in the survey article \cite{rudnicki2002a}. In fact, the same result remains true not only on $L^1$-spaces but on Banach lattices with order continuous norm, 
as observed by the first of the present authors in \cite[Thm 4.2]{gerlach2013b}. Greiner's original theorem was also revisited in a paper by Arendt \cite{arendt2008},
where an application to elliptic differential operators is given. Another related result due to Mokhtar-Kharroubi \cite{mokhtar-Kharroubi2015} focusses on perturbed 
one-parameter semigroups on $L^1$-spaces and was successfully employed in the study of the linear Boltzmann equation in \cite{lods2017}. 
A series of related papers deals with the more specific topic of positive $C_0$-semigroups on atomic Banach lattices, see \cite{davies2005, keicher2006, wolff2008}.

\subsection*{Contributions of this article} In its current stage of development the theory described above has proven useful 
in the asymptotic analysis of evolution equations from various branches of analysis. Yet, it still seems to have a lot of unexplored potential, 
both with regard to applications and from a theoretical point of view. At the same time, the theory currently appears to lack a certain degree of cohesion.

In this paper we present a new and very algebraic approach which allows us to achieve the following three goals:
\begin{enumerate}[\upshape (1)]
	\item We prove much more general versions of the above mentioned results which considerably broadens their range of applicability.
	\item We unify a large part of the theory.
	\item We settle a couple of theoretical questions, thus giving new insight into the foundations of the theory.
\end{enumerate}
Let us elaborate on these points in more detail. 

(1) As a first major generalisation we drop any continuity (or measurability) assumption on the semigroup parameter.
Actually one encounters several classes of one-parameter semigroups in analysis which have rather disparate continuity properties, among them: dual semigroups of $C_0$-semigroups on non-reflexive spaces; semigroups on spaces of measures, which are often only \emph{stochastically}
continuous (cf.\ \cite{hille2009,kunze2009}); and one-parameter semigroups which are strongly continuous on the time interval $(0,\infty)$
but not $C_0$ -- for instance semigroups arising from parabolic equations with certain non-local boundary conditions (see \cite{arendt2016}).
Our results can be applied to all these one-parameter semigroups without any regard to continuity. 

A second important generalisation is that, in case where $T_{t_0}$ only dominates a kernel operator, 
we can replace the irreducibility of the one-parameter semigroup from \cite[Thm~1]{pichor2000} and \cite[Thm 4.2]{gerlach2013b} 
with a much weaker assumption. The usefulness of this will become apparent in the proof of Theorem~\ref{thm:continuous-functions-dominate-compact-operator}. 
Third, we state our main results not only for one-parameter semigroups of the type $(T_t)_{t \in (0,\infty)}$
but for a much more general class of positive semigroup representations. This yields, for instance, also convergence results for multi-parameter semigroups 
of the type $(T_t)_{t \in (0,\infty)^d}$. 

(2) The essence of the present paper are two main results (Theorems~\ref{thm:convergence-AM-compact} and~\ref{thm:convergence-partial-AM-compact}). 
We shall see that one can obtain large parts of the theory described above as simple consequences of these two theorems. 
In our opinion this brings the theory, which currently consists of a variety of disparate results, into a much more cohesive state. 

(3) It is an interesting question what special property of kernel operators enforces the convergence of operator semigroups they are contained in. 
Our approach sheds new light onto this question: it is well-known that, under mild assumptions on the underlying space, 
a kernel operator is a so-called \emph{AM-compact} operator, i.e.\ it maps order intervals to relatively compact sets.
It turns out that this property suffices to develop the entire theory, i.e.\ we can prove our main results for AM-compact operators instead of kernel operators. 
Surprisingly, this leads at the same time to significant simplifications of the proofs. Another interesting theoretical question is what special role 
is played by the time interval $(0,\infty)$.  Simple finite-dimensional counterexamples show that our results do not remain true 
for operator semigroups indexed over the discrete set $\N$, i.e.\ for powers of single operators. We shall find out that this is not because of different topological 
properties of $(0,\infty)$ and $\N$ but rather due to their different algebraic properties, see Example~\ref{ex:dyadic}.

\subsection*{How the article is organised} In the rest of the introduction we briefly comment on related literature and discuss some preliminaries.
In Section~\ref{sec:groupsatomic} we study positive group representations on so-called \emph{atomic} Banach lattices and give sufficient conditions for them to be trivial. 
This is the key to prove our main results about convergence of positive semigroup representations in Section~\ref{sec:convergence}. 
In the subsequent Section~\ref{sec:consequences-i-convergence-results} we demonstrate that many known theorems are corollaries of our main results and in Section~\ref{sec:consequences-ii-spectral-theoretic-results} we discuss some further spectral theoretic consequences.
Section~\ref{sec:consequences-iii-semigroups-on-functions-spaces} contains several consequences on function spaces. Sections~\ref{sec:consequences-i-convergence-results}--\ref{sec:consequences-iii-semigroups-on-functions-spaces} do not aim for outmost generality; the results there are presented for the special case of one-parameter semigroups in order to improve there accessibility for people not interested in the abstract theory and in order to make them comparable to previous results from the literature.

In the Appendices~\ref{appendix:AM-compact} and~\ref{appendix:group-theory} some properties of so-called kernel operators and a few facts from group theory are recalled.

\subsection*{Related results and further literature} Our general assumption in this paper,
which is frequently fulfilled in applications, is that the semigroup representation under consideration contains -- or dominates -- a kernel operator.
Yet, it is worthwhile pointing out that there are also several other approaches to the asymptotic analysis of positive one-parameter semigroups. 
For instance, one can employ results relying on constrictivity properties (see \cite{emelyanov2007} for an overview), 
one can use so-called \emph{lower-bound} methods (see for instance \cite[Thm 5.6.2 and Thm 7.4.1]{lasota1994} and the recent paper \cite{gerlach2016}) 
or one can employ (quasi-)compactness properties \cite[Thm~4]{lotz1986}. 
Another important concept is the spectral theoretic approach which uses spectral properties of the generator of positive $C_0$-semigroups;
we refer to the monographs \cite{nagel1986, batkai2017} for an overview and to the article \cite{mischler2016} for some recent contributions to the theory. 
The asymptotic behaviour of one-parameter  Markov semigroups, which are used to describe Markov processes, can for example be studied by employing the theory of Harris operators 
(see e.g.\ \cite[Ch~V]{foguel1969}) as was for instance done in \cite[Thm 1]{pichor2000}.

\subsection*{Preliminaries}

Let $E,F$ be Banach spaces. We denote by $\cL(E,F)$ the space of bounded linear operators from $E$ to $F$
and write $\cL(E)$ shorthand for $\cL(E,E)$. The dual spaces of $E$ and $F$ are denoted by $E'$ and $F'$ and the adjoint of an operator $T \in \cL(E,F)$ by $T' \in \cL(F',E')$.

Throughout, we assume the reader to be familiar with the theory of Banach lattices; here
we only recall a few things in order to fix the terminology. Unless stated otherwise,
all Banach lattices in this article are real, meaning that the underlying scalar field is $\R$.
Let $E$ and $F$ be Banach lattices. A vector $f \in E$ is called \emph{positive} if $f \ge 0$. We write $f > 0$ if $f \ge 0$ but $f \not= 0$.
The positive cone in $E$ is denoted by $E_+$. For each $f \in E_+$ the \emph{principal ideal} generated by $f$ is the set 
$E_f \coloneqq  \{g \in E: \exists c \ge 0\text{, } \abs{g}\le cf\}$.
A vector $f \in E_+$ is called a \emph{quasi-interior point} of $E_+$ if $E_f$ is dense in $E$. 
If $E = L^p(\Omega,\mu)$ for a $\sigma$-finite measure space $(\Omega,\mu)$ and $p \in [1,\infty)$, 
then a vector $f \in E$ is a quasi-interior point of $E_+$ if and only if $f(\omega) > 0$ for almost all $\omega \in \Omega$. 
For instance, on a finite measure space the principal ideal $E_{\mathds{1}}$, generated by the constant function of value $1$, equals $L^\infty(\Omega,\mu)$. 
Given any subset $A\subseteq E$, we denote by $A^\bot \coloneqq \{ g \in E : \abs{g}\wedge \abs{f} = 0 \; \forall f \in A \}$
the \emph{disjoint complement} of $A$. For $f,g \in E$ we denote the \emph{order interval} between $f$ and $g$ by $[f,g] \coloneqq \{x \in E: \, f \le x \le g\}$. If it is necessary to emphasize the underlying Banach lattice in the notation, we shall write $[f,g]_E$ for $[f,g]$.

We shall often consider Banach lattices with order continuous norm. For a definition and many important properties of these spaces we refer to \cite[Sec II.5]{schaefer1974} 
or to \cite[Sec 2.4]{meyer1991}; here, we only recall that every $L^p$-space (over an arbitrary measure space) has order continuous norm provided that $p \in [1,\infty)$. 
A Banach lattice $E$ is called an \emph{AL-space} if the norm is additive on the positive cone, meaning that $\norm{f+g} = \norm{f} + \norm{g}$ for all $f,g \in E_+$. 
Kakutani's representation theorem for AL-spaces asserts that every AL-space is isometrically lattice isomorphic to $L^1(\Omega,\mu)$ for an appropriate 
(and not necessarily $\sigma$-finite) measure space $(\Omega,\mu)$; see \cite[Thm 2.7.1]{meyer1991}. 

The dual space $E'$ of a Banach lattice $E$ is a Banach lattice, too. 
For every $\varphi\in E'$ we have $\varphi \in E'_+$ if and only if $\applied{ \varphi}{ f} \geq 0$ for all $f \in E_+$.
A functional $\varphi \in E'$ is called \emph{strictly positive} if $\langle \varphi, f \rangle > 0$ for all $0 < f \in E$.

Let $E$ and $F$ be Banach lattices. An operator $T \in \calL(E,F)$ is called \emph{positive} if $TE_+ \subseteq F_+$; it is called \emph{strictly positive} if $Tf > 0$ whenever $f > 0$.
We write $\cL^r(E,F)$
for the regular operators from $E$ to $F$, i.e.\ the linear span of all positive operators in $\cL(E,F)$,
and we set $\cL^r(E)\coloneqq \cL^r(E,E)$.
If $F$ is order complete, then $\cL^r(E,F)$ is itself an order complete vector lattice \cite[Thm 1.3.2]{meyer1991}.
An operator $T \in \calL(E)$ on an AL-space $E$ is called a \emph{Markov operator} if $T$ is positive and norm-preserving on the positive cone, meaning that $Tf \in E_+$ 
and $\norm{Tf} = \norm{f}$ for all $f \in E_+$.

By a \emph{semigroup} $S = (S,\argument)$ we simply mean an algebraic semigroup, 
i.e.\ a set $S$ equipped with an associative binary operation ``$\argument$''.  Let $E$ be a Banach space.
Then the space $\calL(E)$ is a semigroup when endowed with the operator multiplication as binary operation. For a semigroup $S$ 
a semigroup homomorphism $S\to \cL(E)$ is called a \emph{representation of $S$ on $E$} or simply a 
\emph{semigroup representation} on $E$ and is denoted by $(T_t)_{t\in S} \subseteq \cL(E)$.
A semigroup representation is said to be \emph{commutative} if the semigroup $S$ is commutative; 
in this case we denote the semigroup operation on $S$ by ``$+$''.
In the case where $S = \bigl((0,\infty),+\bigr)$ we call a representation $(T_t)_{t\in (0,\infty)}\subseteq \cL(E)$ a \emph{one-parameter semigroup}.
A one-parameter semigroup $(T_t)_{t \in (0,\infty)}$ is called a \emph{$C_0$-semigroup} if $\lim_{t \downarrow 0} T_tx = x$ in $E$ for all $x \in E$. Equivalently, $(T_t)_{t \in (0,\infty)}$ extends to a strongly continuous representation $(T_t)_{t \in [0,\infty)}$ such that $T_0 = \id_E$.
A semigroup representation $(T_t)_{t\in S}$ is
said to be \emph{bounded} if $\sup_{t \in S}\norm{T_t} < \infty$. A subset $M \subseteq E$ is called \emph{invariant} under 
the semigroup representation $\cT = (T_t)_{t \in S}$ if $T_tM \subseteq M$ for all $t \in S$. A vector $x \in E$ is called a \emph{fixed point} 
or a \emph{fixed vector} of $\cT$ if $T_tx = x$ for all $t \in S$; the \emph{fixed space} of $\cT$ is the vector subspace of $E$ that consists of all fixed vectors of $\cT$.
 
Let the space $E$ be a Banach lattice and let $\cT = (T_t)_{t \in S}$ be a semigroup representation on $E$. We say that ``$\cT$ possesses a quasi-interior fixed point'' 
if $\cT$ has a fixed point which is a quasi-interior point of $E_+$. We say that the semigroup representation $\cT$ is positive if each operator $T_t$ is positive. 
A positive semigroup representation $\cT = (T_t)_{t \in (0,\infty)}$ on $E$ is called \emph{irreducible} if $\{0\}$ and $E$ are the only closed ideals in $E$ which are invariant under $\cT$.
A vector $x \in E$ is said to be a \emph{super fixed point} 
of a family of operators $\cT \subseteq \cL(E)$ if $x \ge 0$ and if $Tx\geq x$ for all $T \in \cT$.

If $(G,\cdot)$ is a group (and thus, in particular, a semigroup), then a semigroup representation $(T_t)_{t \in G}$ on a Banach space $E$ 
is called a \emph{group representation} of $G$ on $E$ if each operator $T_t$ is invertible in $\calL(E)$. 
In this case, the mapping $t \mapsto T_t$ is automatically a group homomorphism from $G$ into the group of all invertible operators in $\calL(E)$. A group representation $(T_t)_{t \in G}$ on $E$ is said to be \emph{trivial} if $T_t = \id_E$ for all $t \in G$.

Given a commutative semigroup representation $\cS = (T_t)_{t\in S}$ on a Banach space $E$, we consider $\cS$
as a net where $S$ is directed by 
\begin{align}
	\label{eq:pre-order-on-sg}
	s\leq t \quad \colonequiv \quad t=s \; \text{ or } \; t = s+r \text{ for some } r\in S.
\end{align}
If for given $x, y\in E$ the net $(T_t x)_{t\in S}$ converges to $y$, we denote this by $\lim_{t\in S} T_t x= y$.
In other words, $\lim_{t\in S}T_tx =y$ if and only for every $\eps>0$ there exists $s\in S$ such that $\norm{T_{s+t}x - y}\leq \eps$ for every $t\in S$.
If the net $(T_t x)_{t\in S}$ is convergent in $E$ for every $x \in E$, we call the representation $\cS$ \emph{strongly convergent}.

In many of our results we assume a commutative semigroup $S$ to generate a group, which shall be understood as follows.
A semigroup $S$ is \emph{embeddable into a group} if there exists an injective semigroup homomorphism from $S$ into a group.  If $S$ is commutative,
then this is the case if and only if $S$ is \emph{cancellative}, meaning that $st=sr$ implies $t=r$ for $s,t,r\in S$ \cite[\S 1.10]{clifford1961}.
If a semigroup $S$ is embeddable into a group $G$, we may consider $S$ as a subset of $G$ and we can define the \emph{group generated by $S$}, denoted by $\langle S\rangle$,
as the smallest subgroup of $G$ containing $S$. If, in addition, $S$ is commutative, then the group $\langle S\rangle$ is,
up to group isomorphism, independent of the choice of $G$ and we say that \emph{$S$ generates $\langle S\rangle$}.

A group $(G,\argument)$ is called \emph{divisible} if for every $t \in G$ and every $n \in \N$ there exists an element $s \in G$ such that $s^n = t$. 
For example, the additive groups $\Q$ and $\R$ are divisible, but the group of dyadic numbers $\{k2^{-n}:  k \in \Z\text{, } n \in \N_0\big\}$ is not. A characterisation of divisible groups is given in Proposition~\ref{prop:divisible-groups}.

\section{Group Representations on Atomic Banach Lattices}
\label{sec:groupsatomic}
In this section we study group representations on atomic Banach lattices with order continuous norm. Our main result characterises groups whose bounded positive representations possessing a quasi-interior fixed point are trivial. 

Positive group representations and groups of positive operators have appeared occasionally in the literature 
but it seems that they have not been studied as intensively as, for instance, unitary group representations on Hilbert spaces. 
Some spectral theoretic results about positive group representations can be found in \cite{wolff1982, greiner1983} 
and some structural results about groups of positive operators in \cite[Sec~3]{schaefer1978}. 
A systematic investigation of positive group representations was just initiated in the recent series of papers \cite{jeu2012, jeu2014, jeu2015}. 
While all of these references focus on strongly continuous representations of certain topological groups,
we do not require any topological assumptions in the following.

We start by recalling some basic properties of atomic Banach lattices.
Let $E$ be a Banach lattice. An element $a \in E_+$ is called an \emph{atom} if the principal ideal $E_a$ is one-dimensional. We note that two different atoms of norm $1$ are always disjoint. 
The Banach lattice $E$ is called \emph{atomic} if the band generated by all atoms equals $E$. Such Banach lattices are sometimes also called \emph{discrete}. Typical examples of atomic Banach lattices are real-valued $\ell^p$-spaces (over an arbitrary index set) for $p \in [1,\infty]$, 
as well as the space $c$ of real-valued convergent sequences endowed with the supremum norm and its subspace $c_0$ of sequences which converge to $0$. In these spaces the atoms of norm $1$ are exactly the canonical unit vectors. 

If $E$ is an atomic Banach lattice with order continuous norm and $A$ denotes the set of normalised atoms in $E$, then each $f \in E$ can be represented as an unconditionally convergent series $f = \sum_{a \in A} \lambda_a a$ for appropriate numbers $\lambda_a \in [0,\infty)$, see \cite[p.\,91]{wnuk1999}. We note that, as a consequence of the unconditional convergence, only finitely many of the numbers $\lambda_a$ can be larger than any fixed number $\varepsilon > 0$, and only countably many of them can be non-zero. Those observations imply the following result:

\begin{proposition} \label{prop:atomic-bl-c0}
	Let $E$ be an atomic Banach lattice with order continuous norm and let $A \subseteq E_+$ denote the set of atoms of norm $1$. Then the following assertions hold:
	\begin{enumerate}[\upshape (a)]
		\item There exists a (uniquely determined) lattice homomorphism $J: E \to c_0(A)$ which maps each $a \in A$ to the canonical unit vector $e_a$.
		\item We have $f = \sum_{a \in A} (Jf)_a a$ for each $f \in E$ (where the series converges in the unconditional sense).
	\end{enumerate}
\end{proposition}

The next theorem is the main result of this section. By Proposition~\ref{prop:divisible-groups} a commutative group $G$ fulfils assertion~(i) of the theorem if and only if $G$ is divisible.

\begin{theorem} \label{thm:group-on-atomic-bl}
	For every group $G$ the following assertions are equivalent:
	\begin{enumerate}[\upshape (i)]
		\item Every proper normal subgroup of $G$ has infinite index.
		\item The trivial one is the only positive and bounded representation of $G$ with a quasi-interior fixed point on any atomic Banach lattice with order continuous norm.
	\end{enumerate}
\end{theorem}
\begin{proof}
	(i) $\Rightarrow$ (ii): Assume that (i) holds; let $E$ be an atomic Banach lattice with order continuous norm and let $\cT = (T_t)_{t \in G}$ be a positive and bounded group representation of $G$ with quasi-interior fixed point $y \in E_+$. By Proposition~\ref{prop:atomic-bl-c0} we may assume that $E \subseteq c_0(A)$, where $A$ is set of all atoms of norm $1$ in $E$. Let $\cS(A)$ denote the group of all bijections $A \to A$.

	Since every operator $T_t$ is a lattice isomorphism on $E$, it maps atoms to atoms. Thus, the group representation $(T_t)_{t \in G}$ induces a group homomorphism $\sigma: G \ni t \mapsto \sigma_t \in \cS(A)$ such that $T_t a = \lambda(t,a)\sigma_t(a)$ for all $t \in G$, $a \in A$ and certain $\lambda(t,a) \in (0,\infty)$. As $(T_t)_{t \in G}$ is bounded, there exists $\varepsilon > 0$ such that $\lambda(t,a) \ge \varepsilon$ for all $t \in G$ and all $a \in A$.
	
	Now fix $a \in A$. There exists a number $c > 0$ such that $a \le cy$ and hence, $\lambda(t,a)\sigma_t(a) \le cy$ for each $t \in G$. Since $cy \in c_0(A)$, this implies that the orbit of $a$ under $\sigma$ is finite and hence, by assumption~(i), the orbit consists of $a$ only (see Proposition~\ref{prop:divisible-groups}). Thus, $\sigma_t = \id_A$ for each $t \in G$.
	
	It now follows that for each $a \in A$ the mapping $\lambda(a,\argument): G \to (0,\infty)$ is a bounded group homomorphism, so it is constantly one. Hence, $(T_t)_{t \in G}$ acts trivially on the atoms in $E$ and thus on $E$.
	
	(ii) $\Rightarrow$ (i):
	Assume that (ii) holds and let $N\subseteq G$ be a normal subgroup of finite index; we have to prove that $N = G$.
	
	The factor group $G/N$ operates transitively on itself via, say, left-translation and thus $G$ operates transitively on $G/N$ via the canonical mapping $G \to G/N$. 
	This group action induces a group homomorphism $\psi\colon G \to \cS(G/N)$, where $\cS(G/N)$ denotes the group of all bijections on $G/N$. 
	For every $\pi \in \cS(G/N)$, let $T_\pi$ be the Koopman operator induced by $\pi$ on the finite dimensional space $\ell^2(G/N)$.
	Then $\pi \mapsto T_{\pi^{-1}}$ is a group homomorphism from $\cS(G/N)$ into the group of lattice isomorphisms on $\ell^2(G/N)$
	and the mapping $g \mapsto T_{(\psi(g))^{-1}}$ is a bounded and positive group representation of $G$ on the finite dimensional Banach lattice $\ell^2(G/N)$; 
	the constant function with value $1$ on $G/N$ is a fixed-point of this group representation and a quasi-interior point of $\ell^2(G/N)_+$.
	Hence, the representation $(T_{(\psi(g))^{-1}})_{g \in G}$ is trivial according to (ii), so $\psi(g) = \id_{G/N}$ for each $g \in G$. Since the action of $G$ on $G/N$ is transitive, this implies that $G/N$ is a singleton.
\end{proof}

The implication ``(i) $\Rightarrow$ (ii)'' in the above theorem fails if one considers group representations that 
do not possess a quasi-interior fixed point: let $G$ be the additive group $\R$ and let $E$ be the $\ell^p$-space over the index set $\R$ for some $p \in [1,\infty)$;
then a counteraxample is given by the shift group, given by $(T_tf)(\omega) = f(\omega-t)$ for $f \in E$ and $\omega \in \R$.
In the following we construct a counterexample on an $\ell^p$-space over a countable index set.

\begin{example}
	\label{ex:rotation-group-over-q}
	There exists a non-trivial bounded positive group representation 
	$(T_t)_{t\in \R} \subseteq \cL(E)$ of $(\R,+)$ on $\ell^p(\Q)$, $1 \le p < \infty$.
	
	In order to see this, we first construct a surjective group homomorphism $\phi \colon \R \to \Q$.
	Let $\{ v_j : j\in J \}$ denote a basis of $\R$ as a vector space over $\Q$
	and fix some $i \in J$. Then $V\coloneqq \lin_\Q \{ v_j : j\neq i\}$ 
	is a subspace of $\R$ of codimension $1$. In particular, $V$ is a normal subgroup of $\R$
	and the quotient group $\R/V$ is isomorphic to $\Q$. The composition of the canonical epimorphism $\R \to \R/V$
	with the last-mentioned isomorphism from $\R/V$ to $\Q$ is the desired surjective group homomorphism $\phi$.
	
	For $t\in \R$ we now define $T_t\in \cL(E)$ as the Koopman operator associated with the translation by $\phi(t)$,
	i.e.\ $T_t (x_q)_{q\in \Q} \coloneqq (x_{q+\phi(t)})_{q \in \Q}$ for any $(x_q)\in \ell^p(\Q)$.
	Then $(T_t)_{t \in \R}$ has all the desired properties.
\end{example}

It is worth noting that on an atomic Banach lattice $E$ with order continuous norm every positive, bounded
and \emph{strongly continuous} group representation $(T_t)_{t\in \R} \in\cL(E)$ is trivial. This
follows for instance from \cite[Prop 2.3]{wolff2008}.

\section{Convergence of Positive Semigroup Representations}
\label{sec:convergence}

In this section we prove our main results, convergence theorems for positive semigroup representations.
For this purpose, we need a result from Banach lattice theory which we state in Proposition~\ref{prop:range-of-positive-projection}. 
The latter is essentially contained in \cite[Prop III.11.5]{schaefer1974} but we give a bit more explicit information here.

\begin{proposition} \label{prop:range-of-positive-projection}
	Let $(E,\norm{\argument})$ be a Banach lattice and let $P \in \calL(E)$ be a positive projection. 
	Then $\norm{x}_{PE} \coloneqq \norm{P\abs{x}}$ defines a norm on $PE$ that is equivalent to $\norm{\argument}$
	and $(PE, \norm{\argument}_{PE})$ is a Banach lattice with respect to the order induced by $E$ and with the modulus $\abs{x}_{PE} = P\abs{x}$.
	Moreover, if $y \in E_+$ is a quasi-interior point of $E_+$, then $Py$ is a quasi-interior point of $(PE)_+$.
\end{proposition}
\begin{proof}
	We abbreviate $F \coloneqq PE$. Let $x\in F$.  It follows from $\pm x \leq \abs{x}$ that $\pm x = \pm Px \leq P\abs{x}$. On the other hand for every $z \in F_+$ 
	with $\pm x \leq z$ we have $\abs{x}\leq z$ by the definition of the modulus $\abs{x}$ in $E$. 
	This implies that $P\abs{x}\leq Pz=z$. Hence $P\abs{x}$ is the modulus of $x$ in $F$
	and we proved that $F$ is a vector lattice.

	It is obvious that $\norm{\argument}_{F}$ is a seminorm.
	Since $P\abs{x} \geq \pm x$, we have $P\abs{x}\geq \abs{x}$.  
	This implies that the seminorm $\norm{\argument}_{F}$
	is equivalent to $\norm{\argument}$ on $F$. In particular, $\norm{\argument}_{F}$ is a norm on $F$
	and $(F,\norm{\argument}_{F})$ is a Banach lattice.

	Finally, let $y\in E_+$ be a quasi-interior point of $E_+$. Then
	\begin{align*}
		F = P(\overline{E_y}) \subseteq \overline{PE_y} \subseteq \overline{F_{Py}},
	\end{align*}
	as $P$ is continuous. Hence, $Py$ is a quasi-interior point of $F$.
\end{proof}

As a second ingredient we need a Banach lattice version of a classical result which is sometimes called the \emph{Jacobs--de Leeuw--Glicksberg (JdLG) decomposition theorem}.

\begin{theorem}[JdLG]
\label{thm:JdLG}
	Let $S$ be a commutative semigroup that generates a group $G$
	and let $(T_t)_{t\in S}\subseteq \cL(E)$ be a positive and bounded representation of $S$ on a Banach lattice $E$
	such that the orbits $\{ T_t x : t \in S \}$ are relatively weakly compact for all $x\in E$.

	Then there exists a positive projection $P \in \cL(E)$ that commutes with each $T_t$ and preserves all fixed points of $(T_t)_{t\in S}$
	such that the following assertions hold:
	\begin{enumerate}[\upshape (a)]
	\item There exists a bounded and positive representation $(S_t)_{t\in G}\subseteq \cL(PE)$ of $G$ on $PE$ 
	such that $S_t = T_t|_{PE}$ for each $t\in S$.
	\item For every $x\in \ker P$, $0$ is in the weak closure of $\{ T_t x : t\in S \}$.
	\end{enumerate}
\end{theorem}
\begin{proof}
	Let $\cS$ denote the closure of $\{T_t: \, t \in S\}$ in the weak operator topology. Then $\cS$ is a commutative semigroup of positive operators. The classical Jacobs--de Leeuw--Glicksberg theorem (see e.g.\ \cite[Chapter~16]{eisner2015}) yields a projection $P \in \cS$ such that $\cS$ restricts on $PE$ to a weakly compact group $\cK \subseteq \cL(PE)$ and such that a vector $x \in E$ is contained in $\ker P$ if and only if $0$ is in the weak closure of the orbit $\{T_tx: \, t \in S\}$. By Proposition~\ref{prop:generating-semigroup} the representation $S \ni t \mapsto T_t|_{PE} \in \cK$ can be extended to a group homomorphism $G \to \cK$, which yields the assertion.
\end{proof}

Now we turn to the first of our convergence theorems for certain positive semigroup representations. 
Our main assumption is that the semigroup representation contains a so-called \emph{AM-compact operator}, which is defined as follows:

\begin{definition}
\label{def:AMcompact}
	Let $E$ be a Banach lattice and $Y$ a Banach space.
	A linear operator $T\colon E\to Y$ is called \emph{AM-compact} if $T$ maps order intervals in $E$
	to relatively compact subsets in $Y$.
\end{definition}

\begin{remark}
	\label{rem:atomic-BL-operator-AM-compact}
	It follows from \cite[Cor 21.13]{aliprantis1978} or from \cite[Theorem~6.1]{wnuk1999} that a Banach lattice $E$ is atomic with order continuous norm 
	if and only if every order interval is compact. Since a bounded operator maps relatively compact sets to relatively compact sets, 
	we thus obtain that on an atomic Banach lattice with order continuous norm every bounded operator is AM-compact. 
\end{remark}

As explained in Appendix~\ref{appendix:AM-compact} every integral operator on an $L^p$-space is AM-compact.
This is why the following theorem, our first main result, has a broad range of applications.

\begin{theorem} \label{thm:convergence-AM-compact}
	Let $S$ be a commutative semigroup that generates a divisible group $G = \langle S\rangle$
	and let $\cT = (T_t)_{t\in S}\subseteq \cL(E)$ be a positive and bounded representation of $S$
	on a Banach lattice $E$ such that $T_s $ is AM-compact for some $s\in S$.
	Assume in addition that $\cT$ possesses a quasi-interior fixed point.
	
	Then $\cT$ is strongly convergent.
\end{theorem}

\begin{proof}
	Consider the subsemigroup $S'\coloneqq s + S$ and its positive and bounded representation $\cS \coloneqq (T_t)_{t\in S'}$.
	We show that $\cS$ is strongly convergent; a moment of reflection shows that this implies strong convergence of $\cT$.

	We start by showing that the set $\{T_tx: t\in S' \} \subseteq E$ is relatively compact for every $x\in E$. It suffices (see \cite[Corollary~A.5]{nagel2000}) to consider $x \in E_y$, where $y$ is a quasi-interior $\cT$-fixed point. Then $\abs{x} \le c y$ for some $c > 0$ and therefore $\abs{T_t x} \le c y$ for all $t\in S$. 
	Hence, $T_{s+t}x = T_s T_t x$ is contained in the relatively compact set $T_s[-cy,cy]$ for each $t\in S$.
	
	Therefore, we can apply the Jacobs--de Leeuw--Glicksberg decomposition, Theorem~\ref{thm:JdLG}, to the representation $\cS$ of $S'$.
	Let $P$ be the projection from that theorem and consider the decomposition $E=PE \oplus \ker P$.
	We show next that $\lim_{t\in S'} T_t x=0$ for all $x\in \ker P$. So fix $x\in \ker P$. By assertion (b) of Theorem~\ref{thm:JdLG} there exists 
	a net in $\{T_tx: \; t\in S'\}$ that converges weakly to $0$; this net
	in turn has a subnet that converges to $0$ in norm as the orbit of $x$ is relatively compact.
	As $\cT$ is bounded, this readily implies that $\lim_{t\in S'} T_t x=0$.

	To complete the proof it suffices to show that $T_t |_{PE} = \id_{PE}$ for any $t\in S'$.
	We obtain from Proposition~\ref{prop:range-of-positive-projection} that $F\coloneqq (PE,\norm{\argument}_{PE})$
	is itself a Banach lattice with respect to the order of $E$ and an equivalent norm $\norm{\argument}_{PE}$ and that
	$y= Py$ is a quasi-interior point of $F_+$.
	It is easy to see that $S'$ also generates the group $G$, so we obtain from assertion (a) of Theorem~\ref{thm:JdLG} that there exists a bounded and positive representation $(S_t)_{t\in G}\subseteq \cL(F)$ 
	of $G$ on $F$ such that $S_t = T_t|_{F}$ for each $t\in S'$.

	Finally, we show that $F$ is atomic with order continuous norm. 
	For this, by \cite[Cor 21.13]{aliprantis1978} or \cite[Theorem~6.1]{wnuk1999} it suffices to show that each order interval in $F$ is compact.  Let $[x,z]_F \subseteq F$. We have $S_{2s} = T_{2s}|_F$ as $2s \in S'$ and thus
	\[	[x,z]_F = T_{2s} S_{-2s}[x,z]_F \subseteq T_{2s} [S_{-2s}x, S_{-2s}z]_F \subseteq T_{2s} [S_{-2s}x, S_{-2s}z]_E.\]
	The latter set is relatively compact in $E$ because $T_{2s}$ is AM-compact. 
	Hence, the closed set $[x,z]_F$ is compact in $E$ and thus also in $F$.

	It now follows from Theorem~\ref{thm:group-on-atomic-bl} that every operator $S_t$ is the identity on $F$, so in particular $T_t|_{PE} = \id_{PE}$ for every $t\in S'$.
\end{proof}

\begin{remark} \label{rem:existence-of-quasi-interior-fixed-point}
	In Theorem~\ref{thm:convergence-AM-compact} the existence of a quasi-interior fixed point is crucial: let $\cT = (T_t)_{t \in [0,\infty)}$ be the Gaussian semigroup on $L^1(\R)$; this semigroup is not strongly convergent. However, each operator $T_t$ is AM-compact (cf.\ Proposition~\ref{prop:integraloperatorsAMcompact}), so $\cT$ fulfils all assumptions of Theorem~\ref{thm:convergence-AM-compact} except 
	that it does not possess a quasi-interior fixed point.
\end{remark}

Before we proceed to our second main result, let us point out once again the crucial role of the algebraic structure of the semigroup $S$ in Theorem~\ref{thm:convergence-AM-compact}. Note that the theorem is applicable to the additive group of strictly positive rationals. The situation is different, though, for the dyadic numbers:

\begin{example} \label{ex:dyadic}
	The additive semigroup $D_{> 0} \coloneqq \big\{ k2^{-n}:  k \in \N\text{, } n \in \N_0\big\}$ 
generates the group $D \coloneqq \big\{ k2^{-n}:  k \in \Z\text{, } n \in \N_0\big\}$ which is not divisible. Hence, according to Theorem~\ref{thm:group-on-atomic-bl}
there exists a non-trivial bounded and positive group representation $\cT = (T_t)_{t \in D}$ of $D$ on an atomic Banach lattice with order continuous norm 
and $\cT$ can be chosen to possess a quasi-interior fixed point. 
The restriction of this group representation to $D_{> 0}$ fulfils the assumptions of Theorem~\ref{thm:convergence-AM-compact} but it 
is not strongly convergent by Proposition~\ref{prop:bounded-grp-rep-is-not-convergent} below.
Hence, Theorem~\ref{thm:convergence-AM-compact} fails for semigroup representations of $D_{> 0}$. Since $D_{> 0}$ 
is homeomorphic to the strictly positive rationals, this stresses that the crucial assumption for our results is the algebraic structure 
of the semigroup and not any kind of topological structure.
\end{example}

In the above example we made use of the following simple observation.

\begin{proposition} \label{prop:bounded-grp-rep-is-not-convergent}
	Let $G$ be a commutative group and let $(T_t)_{t \in G}$ be a bounded group representation on a Banach space $E$. 
	Consider a subsemigroup $S$ of $G$ with $\langle S \rangle = G$. If the semigroup representation $(T_t)_{t \in S}$ is strongly convergent, then $T_t = \id_E$ for all $t \in G$. 
\end{proposition}
\begin{proof}
	Assume that $(T_t)_{t \in S}$ is strongly convergent. Let $x \in E$, $t\in S$ and $M \coloneqq \sup_{t \in G} \norm{T_t}$. For any $\eps>0$ we find $s\in S$ such that
	\begin{align*}
		\norm{x-T_tx} = \norm{T_{-s}(T_s x - T_{t+s}x)} \le M \eps.
	\end{align*}
	This show that $T_t = \id_E$ for all $t\in S$ from which the assertion follows.
\end{proof}

Now we come to our second main result: we consider the case where the representation of the semigroup does not necessarily contain an AM-compact operator
but where one operator dominates a non-trivial AM-compact operator. 
Although this condition looks rather technical at a first sight, 
it appears frequently in models from mathematical biology (see for instance \cite{bobrowski2007} for a model from genetics and \cite{du2011} for a model of competing species).

\begin{theorem} \label{thm:convergence-partial-AM-compact}
	Let $S$ be a commutative semigroup that generates a divisible group $G = \langle S\rangle$
	and let $\cT = (T_t)_{t\in S}\subseteq \cL(E)$ be a positive and bounded representation of $S$
	on a Banach lattice $E$ with order continuous norm. Assume that $\cT$ 
	possesses a quasi-interior fixed point $y\in E_+$ and 
	that $\cT$ has the following two properties:
	\begin{enumerate}[\upshape (a)]
		\item Every super fixed point of $\cT$ is a fixed point of $\cT$.
		\item For every fixed point $x > 0$ of $\cT$ there exists $s\in S$ and an AM-compact operator $K \geq 0$ such that $T_s \geq K$ and $Kx>0$.
	\end{enumerate}
	Then $\cT$ is strongly convergent.
\end{theorem}

The key to the proof of Theorem~\ref{thm:convergence-partial-AM-compact} is the following lemma.

\begin{lemma}
\label{lem:singularpartconvergence}
	Under the assumptions of Theorem~\ref{thm:convergence-partial-AM-compact} the following holds:
	for every $\eps>0$ there exists $t\in S$ and an AM-compact
	operator $0\leq K_t \leq T_t$ such that $\norm{(T_t-K_t)y}<\eps$.
\end{lemma}
\begin{proof}
	Since $E$ has order complete norm, the regular AM-compact operators form a band in $\cL^r(E)$, 
	see \cite[Prop 3.7.2]{meyer1991}, which we denote by $\cK$.
	Then for every $t\in S$ we have $T_t = K_t + R_t$ for certain uniquely determined positive operators $K_t \in \cK$ and $R_t \in \cK^\bot$. 
	Moreover,
	\[ T_{t+r} = K_tT_r + R_t T_r = K_tT_r + R_t K_r + R_t R_r\]
	for all $t,r\in S$.
	Since $K_tT_r+ R_tK_r$ is AM-compact and dominated by $T_{t+r}$,
	we have $K_{t+r} \ge K_tT_r + R_tK_r$. This in turn implies that $R_{t+r} \leq R_tR_r$.
	In particular, $R_{t+r}y \leq R_t T_ry = R_t y$ for all $t,r\in S$, i.e.\
	the net $(R_ty)_{t\in S} \subseteq E_+$ is decreasing and thus convergent as the norm is order continuous.
	Its limit $z\coloneqq \lim_{t\in S} R_ty$ fulfils
	\[ R_t z  = \lim_{r \in S} R_tR_r y \geq \lim_{r\in S} R_{t+r} y = z\]
	and hence $T_tz \geq R_t z \geq z$ for every $t\in S$. 
	By assumption (a) this implies that $T_tz=z$ and therefore $K_t z=0$ for every $t\in S$.
	Assumption~(b) implies that $z=0$. We have thus shown that $\lim_{t \in S} R_ty = 0$, which proves the lemma.
\end{proof}

\begin{proof}[Proof of Theorem~\ref{thm:convergence-partial-AM-compact}]
	We first show that $\{ T_t x : t\in S\}\subseteq E$ is relatively weakly compact for every $x\in E$.
	First let $x \in E_{y}$, i.e.\ $x \in [-cy,cy]$ for some $c > 0$; then $T_tx \in [-cy,cy]$ for all $t \in S$. 
	By the order continuity of the norm every order interval in $E$ is weakly compact and therefore $\{T_t x: t\in S\}$
	is relatively weakly compact.  Since $E_{y}$ is dense in $E$ and the representation is bounded, 
	it follows that the orbit of any $x\in E$ is relatively weakly compact (see for instance \cite[Corollary~A.5]{nagel2000}).
	
	We can thus apply Theorem~\ref{thm:JdLG}; let $P$ be the projection given by this theorem.	
	We are going to show first that $\lim T_tx=0$ for any $x\in \ker P$ and second that $T_t|_{PE} = \id_{PE}$ for all $t\in S$.
	By combining these two assertions we obtain the theorem.

	Let $z \in [0,y]$ and consider the vector $x \coloneqq z-Pz \in \ker P$.
	Since $Pz\leq Py =y$ we have $x \in [-y, y]$, i.e.\ $\abs{x} \leq y$.
	By assertion (b) of Theorem~\ref{thm:JdLG} there exists a net $(x_\alpha)_{\alpha\in\Lambda} \subseteq \{ T_t x : t\in S\}$
	such that $(x_\alpha)_{\alpha\in\Lambda}$ converges weakly to $0$. 
	Let $\eps > 0$. By Lemma~\ref{lem:singularpartconvergence} we find $r\in S$ and an AM-compact operator $0\leq K_r \leq T_r$
	such that $\norm{T_r y- K_ry} <\eps$.
	For every $\alpha \in \Lambda$ we thus have
	\begin{displaymath}
		\norm{T_r x_\alpha} \leq \norm{K_r x_\alpha} + \norm{(T_r-K_r) x_\alpha} 
		\leq \norm{K_r x_\alpha} + \varepsilon,
	\end{displaymath}
	where the second inequality follows from the fact that $\abs{(T_r-K_r) x_\alpha} \leq (T_r-K_r)y$.
	The net $(K_r x_\alpha)$ is contained in the relatively compact set $K_r[-y,y]$, so there exists a subnet of $(K_rx_\alpha)_{\alpha\in \Lambda}$
	that converges in norm; since it also converges weakly to $0$, it follows that $\norm{K_r x_\alpha} < \eps$
	for some $\alpha\in\Lambda$. In summary, this shows that $\norm{T_t x} < 2\eps$ for some $t\in S$.
	Since $\cT$ is bounded and $y$ is a quasi-interior point of $E_+$, this readily implies 
	that $\lim_{t\in S} T_t  x = 0$ for all $x\in \ker P=(\id_E-P)E$.

	Now we show that each $T_t$ acts trivially on $PE$. By Proposition~\ref{prop:range-of-positive-projection},
	$F\coloneqq(PE,\norm{\argument}_{PE})$ is itself a Banach lattice with respect to the order of $E$
	and an equivalent norm $\norm{\argument}_{PE}$; moreover, $y= Py$ is a quasi-interior point of $F_+$. We obtain from assertion (a) of Theorem~\ref{thm:JdLG} that there exists a bounded and positive representation $(S_t)_{t\in G}\subseteq \cL(F)$ 
	of $G$ on $F$ such that $S_t = T_t|_{F}$ for each $t\in S$.
	 
	Let $u,v \in F$ such that $u\leq v$.  We show that $[u,v]_F$ is totally bounded and thus compact in $F$. Given $\eps > 0$,
	first choose $c>0$ such that $[u,v]_F \subseteq [-cy,cy]_F+B_E(0,\eps)$, where $B_E(0,\eps)$ denotes the ball in $E$ of radius $\eps$
	centered at $0$. 
	Then, by Lemma~\ref{lem:singularpartconvergence}, we find $r\in S$ and an AM-compact operator $0\leq K_r\leq T_r$ 
	such that $\norm{(T_r-K_r)cy}<\eps$.  Thus we have
	\begin{align*}
		[u,v]_F &\subseteq [-cy,cy]_F + B_E(0,\eps) = T_rS_{-r}[-cy,cy]_F + B_E(0,\eps) \\
		&= T_r[-cS_{-r}y,cS_{-r}y]_F +B_E(0,\eps)  \subseteq T_r [-cy,cy]_E +B_E(0,\eps)  \\
		&\subseteq K_r[-cy,cy]_E + (T_r-K_r)[-cy,cy]_E + B_E(0,\eps) \\
		&\subseteq K_r[-cy,cy]_E + B_E(0,2\eps).
	\end{align*}
	Since $K_r$ is AM-compact, the set $K_r[-cy,cy]_E$ is totally bounded in $E$ and hence, so is $[u,v]_F$. Therefore, $[u,v]_F$ is compact in $E$ and thus in $F$.
	This shows that $F$ is an atomic Banach lattice with order continuous norm, cf.\ \cite[Cor 21.13]{aliprantis1978} or \cite[Theorem~6.1]{wnuk1999}.

	As $G$ is divisible and the quasi-interior point $y\in F_+$ is a fixed vector of the group representation $(S_t)_{t \in G}$,
	it follows from Theorem~\ref{thm:group-on-atomic-bl} that every operator $S_t$ is the identity on $F$. 
	In particular, $T_t|_{PE} = \id_{PE}$ for every $t\in S$ which completes the proof.
\end{proof}

We close the section with a discussion of the technical requirements (a) and (b) of Theorem~\ref{thm:convergence-partial-AM-compact}. 
The following proposition provides sufficient conditions for assumption (a) to hold.

\begin{proposition}
\label{prop:conditions-in-convergence-theorem}
	Let $E$ be a Banach lattice and let $\cT = (T_t)_{t\in S}$ be a positive and bounded 
	representation of a semigroup $S$ on $E$.
	Each of the following conditions implies that every super fixed point of $\cT$ is a fixed point of $\cT$ 
	and that the fixed space of $\cT$ is a sublattice of $E$.
	\begin{enumerate}[\upshape (a)]
		\item $E$ has strictly monotone norm, meaning that $\norm{f} < \norm{g}$ whenever $0 \le f < g$, and each operator in  $\cT$ is contractive.
		\item There exists a strictly positive $\phi \in E'_+$ such that $T_t' \phi \le \phi$ for all $t\in S$.
		\item The representation $\cT$ is irreducible.
		\item The space $E$ has order continuous norm and the adjoint $T_t'$ of each operator $T_t$ is a lattice homomorphism. 
	\end{enumerate}
\end{proposition}
\begin{proof}
	First note that if $x\in E$ is a fixed point of $\cT$, then $\abs{x} = \abs{T_tx} \leq T_t\abs{x}$ for every $t\in S$,
	i.e.\ $\abs{x}$ is a super fixed point of $\cT$. This shows that
	the fixed space of $\cT$ is a sublattice of $E$ if every super fixed point of $\cT$ is a fixed point.

	(a) Fix $t\in S$ and let $x\in E_+$ such that $T_t x\geq x$. Since $T_t$ is contractive we have 
	$\norm{T_tx}=\norm{x}$ and as the norm is strictly monotone, this readily implies that $T_tx=x$.

	(b) Fix $t\in S$ and let $x\in E_+$ such that $T_t x\geq x$; then $\applied{T_tx-x}{\phi}\geq 0$. On the other hand,
	$\applied{T_tx-x}{\phi} = \applied{x}{T_t'\phi - \phi}\leq 0$ and thus $\phi$ vanishes on the positive vector $T_tx-x$. 
	Since $\phi$ is strictly positive, $T_tx-x=0$.
	
	(c) Let $x\in E_+$ be a non-zero super-fixed point of $\cT$ and fix a functional $\alpha \in E'_+$ such that $\applied{\alpha}{x}>0$. Let $K \subseteq E'$ denote the weak${}^*$-closure of the convex hull of the orbit $\{T_t'\alpha: \, t \in S\}$. It follows from Day's fixed point theorem \cite[Theorem~3]{day1961} that the dual semigroup $\cT' \coloneqq (T_t')_{t \in S}$ has a fixed point $\phi \in K$. In particular, $\phi$ is positive, and non-zero since $\langle \phi, x\rangle \ge \langle \alpha,x\rangle > 0$. As $\cT$ is irreducible, $\phi$ is even strictly positive. Hence, assumption~(b) is fulfilled and this proves the assertion.
	
	(d) Fix $t\in S$ and let $x\in E_+$ such that $T_tx\geq x$. Since $T_t'$ is a lattice homomorphism and the norm on $E$
	is order continuous, it follows from \cite[Exer 1.4.E2]{meyer1991} that $T_t$ is interval preserving.
	Thus, we have $T_t[0,x]=[0,T_tx] \supseteq [0,x]$ and hence we find $x_1\in[0,x]$ such that $T_tx_1=x$. Now we construct recursively
	a decreasing sequence $(x_n)_{n \in \N_0} \subseteq E_+$ such that $T_tx_{n+1} = x_n$ and $x_0=x$.
	Since $E$ has order continuous norm, the sequence $(x_n)$ converges and $z\coloneqq \lim x_n$ fulfils
	$T_tz = z$.  Now it follows from 
	\[ \norm{ z-x } = \norm{T_t^n z - T_t^n x_n} \leq \sup_{s\in S} \norm{T_s} \cdot \norm{z-x_n} \to 0 \text{ as } n\to\infty \]
	that $x=z$ is a fixed point of $T_t$. 
\end{proof}

\begin{remark} \label{rem:sufficient-for-domination}
Regarding assumption (b) of Theorem~\ref{thm:convergence-partial-AM-compact}, it is quite obvious 
that in order to enforce convergence it cannot be sufficient that some operator $T_t$ simply dominates a non-trivial AM-compact operator.
Instead, one has to ensure that the dominated AM-compact operators interact, in a sense, with the entire semigroup. 
Assumption~(b) in Theorem~\ref{thm:one-parameter-dominate-kernel} shows that this ``interaction condition'' 
is as weak as one could possibly hope for: it suffices that the family of dominated AM-compact operators 
sees every positive fixed vector of the semigroup. 
\end{remark}

\section{Convergence of One-Parameter Semigroups}
\label{sec:consequences-i-convergence-results}

In this section we present some special cases of Theorems~\ref{thm:convergence-AM-compact} and~\ref{thm:convergence-partial-AM-compact}
which are most important for applications or interesting in their own right. Many of them are generalisations of known theorems
in so far as any continuity requirement is dropped. 

\subsection*{Semigroups of kernel operators}

In the following we consider semigroups which contain a so-called \emph{kernel operator}. For details about such operators we refer to Appendix~\ref{appendix:AM-compact}. Here we only recall that every kernel operator is AM-compact (Proposition~\ref{prop:kernel-AM-compact}), so we can apply our results from Section~\ref{sec:convergence}. Let us start with the following generalisation of Greiner's theorem quoted in the introduction.

\begin{theorem}
	\label{thm:convergence-kernel-operators}
	Let $E$ be a Banach lattice with order continuous norm and let $\cT = (T_t)_{t \in (0,\infty)}$ be a bounded and
	positive one-parameter semigroup on $E$. Assume that for some $s >0$ the operator $T_s$ is a kernel operator and 
	that $\cT$ possesses a quasi-interior fixed point.
	
	Then $\cT$ is strongly convergent.
\end{theorem}
\begin{proof}
	This follows from Theorem~\ref{thm:convergence-AM-compact} and Proposition~\ref{prop:kernel-AM-compact}.
\end{proof}

As explained in the introduction, Greiner proved this result for the special case of contractive $C_0$-semigroups on $L^p$-spaces 
in \cite[Kor 3.11]{greiner1982} (see also \cite[Thm 12]{davies2005} for a related spectral result)
and deduced it from a certain $0$-$2$-law (\cite[Thm 3.7]{greiner1982}; see also \cite[Thm 5.1]{gerlach2013b}),
which has itself a very technical proof. 
By contrast, our proof of Theorem~\ref{thm:convergence-AM-compact} only relies on the fact that 
every kernel operator is AM-compact, 
on the Jacobs--de Leeuw--Glicksberg decomposition and on our result on group representations, Theorem~\ref{thm:group-on-atomic-bl}.

On atomic Banach lattices with order continuous norm every positive operator is a kernel operator 
(see e.g.\ \cite[Lem~4.1.5]{gerlach2014b}). Hence, we obtain the following special case of Theorem~\ref{thm:convergence-kernel-operators}.

\begin{theorem}
	\label{thm:convergence-atomic-Banach-lattice}
	Let $E$ be an atomic Banach lattice with order continuous norm and let $\cT = (T_t)_{t \in (0,\infty)}$ be a bounded and
	positive one-parameter semigroup on $E$. 
	If $\cT$ possesses a quasi-interior fixed point, then $\cT$ is strongly convergent.
\end{theorem}

Instead of using \cite[Lem~4.1.5]{gerlach2014b}, one can also derive this result directly from Theorem~\ref{thm:convergence-AM-compact} since every operator on $E$ is AM-compact according to Remark~\ref{rem:atomic-BL-operator-AM-compact}.

For $C_0$-semigroups Theorem~\ref{thm:convergence-atomic-Banach-lattice} was proved by Keicher \cite[Cor 3.8]{keicher2006}; 
see also \cite{davies2005} and \cite{wolff2008} for related results about positive $C_0$-semigroups on atomic Banach lattices.

\subsection*{Semigroups that dominate kernel operators}

Now we consider semigroups which do not necessarily contain a kernel operator but which at least dominate a kernel operator. 

\begin{theorem} \label{thm:one-parameter-dominate-kernel}
	Let $E$ be a Banach lattice with order continuous norm and let $\cT = (T_t)_{t \in (0,\infty)}$ be a bounded and
	positive one-parameter semigroup on $E$. Assume that $\cT$ possesses a quasi-interior fixed point
	and that the following two assumptions are fulfilled:
	\begin{enumerate}[\upshape (a)]
		\item Every super fixed point of $\cT$ is a fixed point.
		\item For every fixed point $x > 0$ of $\cT$ there exists $s > 0$ and a kernel operator $0 \le K \le T_s$ which fulfils $Kx > 0$.
	\end{enumerate}
	Then $\cT$ is strongly convergent.
\end{theorem}
\begin{proof}
	This follows from Theorem~\ref{thm:convergence-partial-AM-compact} and Proposition~\ref{prop:kernel-AM-compact}.
\end{proof}

Assumption~(b) is in particular fulfilled if, for at least one time $s \in (0,\infty)$, there is a kernel operator $0 \le K \le T_s$ which maps non-zero positive vector again to non-zero vectors. On $L^1$-spaces such a condition also occurs in two recent papers of Pich\'{o}r and Rudnicki \cite{pichor2016, pichor2017}; see the conditions~(K) in \cite[pp.~308 and~309]{pichor2016} and in the introduction of~\cite{pichor2017}.

For irreducible semigroups Theorem~\ref{thm:one-parameter-dominate-kernel} takes a simpler form.

\begin{corollary}
	\label{cor:one-parameter-dominate-kernel:irreducible}
	Let $E$ be a Banach lattice with order continuous norm and let $\cT = (T_t)_{t \in (0,\infty)}$ be a bounded and
	positive one-parameter semigroup on $E$ with a non-zero fixed point.
	
	If $\cT$ is irreducible and if $T_s \ge K \ge 0$ for some $s > 0$ and for some non-zero kernel operator $K$, then $\cT$ is strongly convergent.
\end{corollary}
\begin{proof}
	By Proposition~\ref{prop:conditions-in-convergence-theorem}, 
	the irreducibility of $\cT$ implies that the the fixed space of $\cT$ is a sublattice of $E$. 
	Hence, $\cT$ also possesses a positive non-zero fixed vector $y$. 
	Employing again the irreducibility assumption we see that $y$ is even a quasi-interior point of $E_+$. 
	Thus, the assertion follows from Theorem~\ref{thm:one-parameter-dominate-kernel}, where
	assumption~(a) is fulfilled by Proposition~\ref{prop:conditions-in-convergence-theorem}(c) and assumption~(b) 
	is fulfilled since every non-zero positive fixed point of $\cT$ is a quasi-interior point.
\end{proof}

For $C_0$-semigroups the above result was proved by the first author in \cite[Thm 4.2]{gerlach2013b} by using Greiner's $0$-$2$-law. 
For strongly continuous Markov semigroups on $L^1$-spaces, Corollary~\ref{cor:one-parameter-dominate-kernel:irreducible} was proved 
earlier by Pich\'{o}r and Rudnicki \cite[Thm~1]{pichor2000} who used the theory of Harris operators on $L^1$-spaces; this theory is, for instance, presented in \cite[Ch~V]{foguel1969}.

\begin{remark}
	For perturbed $C_0$-semigroups, that assumption from Corollary~\ref{cor:one-parameter-dominate-kernel:irreducible} that an operator $T_s$ dominates a non-zero kernel operator can sometimes be checked by employing the Dyson--Philipps series expansion. For a concrete example we refer to \cite[Subsection~2.3]{pichor2000}.
\end{remark}

\section{Spectral Theoretic Consequences} \label{sec:consequences-ii-spectral-theoretic-results}

In this section we are concerned with one-parameter semigroups of kernel operators that do not
necessarily possess a quasi-interior fixed point. As we have seen in Remark~\ref{rem:existence-of-quasi-interior-fixed-point},
we cannot expect them to be strongly convergent in general. Instead, we show that they satisfy a certain
spectral condition which is necessary for strong convergence.  To make this precise, we introduce the following terminology. 

Let $\cT = (T_t)_{t \in (0,\infty)}$ be a one-parameter semigroup on a complex Banach space $E$. A function $(0,\infty) \ni t \mapsto \lambda_t \in \C$,
also denoted by $(\lambda_t)_{t\in (0,\infty)}$,
is called an \emph{eigenvalue} of $\cT$ if there exists a vector $x \in E \setminus \{0\}$, called a corresponding \emph{eigenvector},
such that $T_t x = \lambda_t x$ for all $t \in (0,\infty)$.  An eigenvalue $(\lambda_t)_{t \in (0,\infty)}$ of the semigroup $\cT$ is called \emph{unimodular} if $\abs{\lambda_t} = 1$ for all $t \in (0,\infty)$. 
Now, let $\cT = (T_t)_{t \in (0,\infty)}$ be a one-parameter semigroup on a real Banach lattice $E$ and denote by $E_\C$ the complexification of $E$.
Each operator $T_t$ admits a canonical $\C$-linear extension $E_\C \to E_\C$ which we denote again by $T_t$ for simplicity.
We say that $(\lambda_t)_{t \in (0,\infty)} \subseteq \C$ is an eigenvalue of $\cT$ if it is an eigenvalue of the 
one-parameter semigroup consisting of the complex extensions of the operators $T_t$.

\begin{remark} \label{rem:eigenvalues-for-c_0-semigroups}
	Let $\cT = (T_t)_{t \in (0,\infty)}$ be a $C_0$-semigroup with generator $A$ on a complex Banach space $E$ and let $\lambda = (\lambda_t)_{t \in (0,\infty)}$ 
	be a complex-valued function. 
	Then $\lambda$ is an eigenvalue for $\cT$ if and only if there exists an eigenvalue $\mu$ of $A$ such that $\lambda_t = e^{t\mu}$ for all $t \in (0,\infty)$.
\end{remark}
\begin{proof}
	$\Leftarrow$: If $0 \not= x \in \ker(\mu - A)$, then $(z - A)^{-1}x = (z-\mu)^{-1}x$ for all $z$ within the resolvent set of $A$. 
	By the Euler formula for $C_0$-semigroups \cite[Cor~III.5.5]{nagel2000} this implies that $T_tx = e^{\mu t}x$ for all $t \in (0,\infty)$.
	
	$\Rightarrow$: Let $x \in E \setminus \{0\}$ such that $T_tx = \lambda_t x$ for all $t \in (0,\infty)$. 
	With the definition $\lambda_0 \coloneqq 1$ it follows from the semigroup law that $\lambda_{t+s} = \lambda_t \lambda_s$ for all $t,s \in [0,\infty)$.
	Moreover, the mapping $[0,\infty) \ni t \mapsto \lambda_t \in \C$ is continuous since $\cT$ is a $C_0$-semigroup. 
	According to \cite[Thm~I.1.4]{nagel2000} this implies that there exists a number $\mu \in \C$ such that $\lambda_t = e^{\mu t}$ for all $t \in [0,\infty)$. 
	It now follows from the very definition of the generator $A$ that $\mu$ is an eigenvalue of $A$ with eigenvector $x$.
\end{proof}

We now prove criteria for the absence of non-trivial unimodular eigenvalues. 
This property is of interest as it is closely related to the asymptotic behaviour of the semigroup. Let us briefly illustrate this with the following proposition:

\begin{proposition} \label{prop:existence-of-periodic-orbits}
	Let $\cT = (T_t)_{t \in (0,\infty)}$ be a one-parameter semigroup in a real or complex Banach space $E$ which does not possess any unimodular eigenvalue except possibly $(1)_{t \in (0,\infty)}$.
	\begin{enumerate}[\upshape (a)]
		\item If $\cT$ has relatively compact orbits, then $\cT$ is strongly convergent.
		\item If the mapping $(0,\infty) \ni t \mapsto T_t \in \cL(E)$ is strongly continuous, then $\cT$ does not admit non-trivial periodic orbits.
	\end{enumerate}
\end{proposition}

Here, way say that $\cT$ admits a \emph{non-trivial periodic orbit} if there exists a vector $x \in E$ and a time $t_0 > 0$ such that $T_t x = T_{t_0 + t}x$ 
for all $t \in (0,\infty)$ while the mapping $(0,\infty) \ni t \mapsto T_tx \in E$ is not constant; in this case, $t_0$ is called a \emph{period} of this orbit.

\begin{proof}[Proof of Proposition~\ref{prop:existence-of-periodic-orbits}]
	By passing to a complexification, we may assume that the scalar field is $\C$.
	
	(a) This is a consequence of \cite[Thm~4.5 in \S2]{krengel1985}.
	
	(b) Suppose that the orbit $(0,\infty) \ni t \mapsto T_t x \in E$ is not constant but periodic with $t_0 > 0$ as a period. 
	Denote by $F \subseteq E$ the closed linear span of $\{T_tx: t \in (0,\infty)\}$. Then $F$ is non-zero and invariant with respect to $\cT$. 
	The restriction $\cS \coloneqq (T_t|_F)_{t \in (0,\infty)}$ of $\cT$ to $F$ is periodic with $t_0$ as a period, but it is not constant 
	since $t \mapsto T_t|_F T_{t_0}x$ is not constant.
	Moreover, $T_{t_0}|_F$ is the identity operator on $F$, so $\cS$ extends to a periodic $C_0$-semigroup on $F$. 
	We conclude from \cite[Thm~IV.2.26]{nagel2000} that the generator of $\cS$ possesses an eigenvalue $i\beta \in i\R \setminus \{0\}$, 
	so it follows from Remark~\ref{rem:eigenvalues-for-c_0-semigroups} that $(e^{it\beta})_{t \in (0,\infty)}$ is an eigenvalue of $\cS$ and hence of $\cT$.
\end{proof}

We point out that assertion~(b) of the above proposition fails if one drops the strong continuity assumption. 
A counterexample is again provided by the shift semigroup on an $\ell^p$-space over the index set $\R$.

Now we give sufficient criteria for a positive semigroup to have no non-trivial unimodular eigenvalue. 
Recall that a Banach lattice $E$ is called a \emph{KB-space} if every norm-bounded increasing net (or equivalently: sequence) in $E$ converges. 
Important examples of KB-spaces are reflexive Banach lattices and AL-spaces.

\begin{theorem} \label{thm:eigenvalues-AM-compact}
	Let $\cT = (T_t)_{t\in (0,\infty)}$ be a bounded and positive one-parameter semigroup on a Banach lattice $E$
	such that $T_s$ is AM-compact for some $s\in S$.
	
	If every super fixed point of $\cT$ is a fixed point of $\cT$ 
	or if $E$ is a KB-space, then $(1)_{t\in (0,\infty)}$ is the only possible unimodular eigenvalue of $\cT$.
\end{theorem}
\begin{proof}
	Assume that $(\lambda_t)_{t \in (0,\infty)}$ is a unimodular eigenvalue for $\cT$. 
	Denote by $z \in E_\C \setminus \{0\}$ a corresponding eigenvector, where $E_\C$ is the Banach lattice complexification of $E$. 
	We have to show that $\lambda_t = 1$ for all $t \in (0,\infty)$.
	
	We first prove that there exists a fixed point $f \in E_+$ for $\cT$ which fulfils $f \ge \abs{z}$. 
	Since $\abs{z} = \abs{\lambda_t z} = \abs{T_t z} \le T_t \abs{z}$ for each $t \in (0,\infty)$, $\abs{z}$ is a super-fixed point of $\cT$. 
	Hence, under the  assumption that every super fixed point of $\cT$ is a fixed point we have found our fixed point $f \coloneqq \abs{z}$. 
	If instead $E$ is a KB-space, the increasing and norm bounded net $(T_t\abs{z})_{t \in (0,\infty)} \subseteq E$ converges to a vector $f \in E_+$,
	which is clearly a fixed point of $\cT$ and dominates $\abs{z}$.
	
	Let $F$ denote the closure of the principal ideal $E_{f}$ in $E$ and let $F_\C \coloneqq F + iF \subseteq E_\C$. 
	Since $E_{f}$ is $\cT$-invariant, so are $F$ and $F_\C$. The restriction of $\cT$ to $F$ has $f$ as a quasi-interior fixed point and the restriction of $T_s$ 
	to $F$ is AM-compact. Thus, $\cT$ fulfils all assumptions of Theorem~\ref{thm:convergence-AM-compact}, which implies that
	that $\lim_{t \to \infty} T_tg$ exists for all $g \in F$ and hence also for all $g \in F_\C$. 
	In particular, $(T_t z)$ converges as $t$ tends to infinity. Hence, if we fix $t \in (0,\infty)$, then the 
	sequence $(T_{nt}z) = (\lambda_t^n z)$ converges.
	This implies $\lambda_t = 1$.
\end{proof}

If the semigroup only dominates an AM-compact operator, 
we have the following result, which is a consequence of Theorem~\ref{thm:convergence-partial-AM-compact}. 
The proof is very similar to that of Theorem~\ref{thm:eigenvalues-AM-compact}.

\begin{theorem} \label{thm:eigenvalues-partial-AM-compact}
	Let $\cT = (T_t)_{t\in (0,\infty)}$ be a bounded and positive one-parameter semigroup 
 	on a Banach lattice $E$ with order continuous norm. Assume that $\cT$ has the following two properties:
	\begin{enumerate}[\upshape (a)]
		\item Every super fixed point of $\cT$ is a fixed point of $\cT$.
		\item For every fixed point $x > 0$ of $\cT$ there exists $s\in (0,\infty)$ and an AM-compact operator $K \geq 0$ such that $T_s \geq K$ and $Kx>0$.
	\end{enumerate}
	Then  $(1)_{t \in S}$ is the only possible unimodular eigenvalue of $\cT$.
\end{theorem}
\begin{proof}
	Assume that $(\lambda_t)_{t \in (0,\infty)}$ is a unimodular eigenvalue for $\cT$, 
	denote by $E_\C$ the Banach lattice complexification of $E$ and let $z \in E_\C \setminus \{0\}$ be an eigenvector corresponding to $(\lambda_t)_{t \in (0,\infty)}$. 
	As in the proof of Theorem~\ref{thm:eigenvalues-AM-compact}, $\abs{z}$ is a fixed point of $\cT$.
	Let $F$ denote the closure of the principal ideal $E_{\abs{z}}$ and set $F_\C \coloneqq F + i F$. Then $F$, and thus $F_\C$, is $\cT$-invariant
	and it is easy to verify that the restriction of $\cT$ to $F$ satisfies all assumptions of Theorem~\ref{thm:convergence-partial-AM-compact}.
	Thus, $\lim_{t \to \infty} T_tg$ exists for all $g \in F$ and hence also for all $g \in F_\C$. 
	In particular, $(T_tz)$ converges as $t$ tends to infinity.
	Thus, $\lambda_t = 1$ for all $t \in (0,\infty)$.
\end{proof}

\begin{remark}
	Theorems~\ref{thm:eigenvalues-AM-compact} and~\ref{thm:eigenvalues-partial-AM-compact} remain true if we replace $(T_t)_{t \in (0,\infty)}$ with a representation of an arbitrary commutative semigroup (where one has to use the obvious generalisation of the notion \emph{eigenvalue}). The proofs are the same except for the last step where one has to employ a slightly more involved argument since a commutative semigroup need not be Archimedean, in general. For instance, the following observation does the job:
	
	\textit{If $t \mapsto \lambda_t$ is a semigroup homomorphism from a commutative semigroup $S$ into the complex unit circle and if the net $(\lambda_t)_{t \in S}$ converges, then $\lambda_t = 1$ for all $t \in S$.}
	
	To see this, first note that $\lim_t \lambda_t = 1$. Let $\varepsilon > 0$ and choose $t_0 \in S$ such that $|\lambda_t - 1| < \varepsilon$ for all $t \ge t_0$. Then, $|\lambda_{t_0} - 1| < \varepsilon$ and $|\lambda_{t_0 + t} - 1| < \varepsilon$ for all $t \in S$. We thus conclude that 
	\begin{align*}
		|\lambda_t - 1| = |\lambda_{t_0+t} - \lambda_{t_0}| < 2\varepsilon
	\end{align*}
	for all $t \in S$. Since $\varepsilon > 0$ was arbitrary, all $\lambda_t$ are equal to $1$.
\end{remark}

\section{Semigroups on function spaces}
\label{sec:consequences-iii-semigroups-on-functions-spaces}

In this final section we consider $L^p$-spaces and spaces of continuous functions. We discuss consequences of our main results for positive one-parameter semigroups on those spaces, in particular under the assumptions that the space contains an atom (Theorem~\ref{thm:semigroup-on-measure-space-with-atom} and Corollary~\ref{cor:continuous-functions-isolated}).

\subsection*{Semigroups over measure spaces with atoms}

In the following we briefly discuss semigroups over $L^p$-spaces in case that the underlying measure space contains an atom. Here, a point $\omega_0$ in a $\sigma$-finite measure space $(\Omega,\mu)$ is called an \emph{atom} if $\{\omega_0\}$ is measurable and $\mu(\{\omega_0\}) > 0$. This simple assumption has astonishing consequences for the asymptotic behaviour of the semigroup.

\begin{theorem}
	\label{thm:semigroup-on-measure-space-with-atom}
	Let $(\Omega,\mu)$ be a $\sigma$-finite measure space with an atom $\omega_0 \in \Omega$. Let $p \in [1,\infty)$ and let $\cT = (T_t)_{t \in (0,\infty)}$ be 
	a bounded and positive semigroup on $E \coloneqq L^p(\Omega,\mu)$. If $\cT$ is irreducible and has a non-zero fixed-point, then $\cT$ is strongly convergent.
\end{theorem}
\begin{proof}
	The assertion is obvious if $\dim E=1$, so assume that $\dim E \ge 2$. 
	Let $e\colon \Omega \to \R$ be the indicator function of the singleton $\{\omega_0\}$. As $(\Omega,\mu)$ is $\sigma$-finite we have $\mu(\{\omega_0\}) < \infty$, so $e \in E$.
	Denote by $B\coloneqq \{e\}^{\bot\bot}$ the band generated by $e$ and by $P \in \cL(E)$ the band projection onto $B$.
	Since $B$ is one-dimensional and $\dim E \ge 2$, the orthogonal band $B^\bot$ is non-zero and can therefore not be invariant under the semigroup $\cT$. 
	Thus, there exists $t \in (0,\infty)$ such that $PT_t \not= 0$. On the other hand, $PT_t$ has rank $1$ and is thus a kernel operator. 
	Since $T_t$ dominates $PT_t$, the assertion follows from Corollary~\ref{cor:one-parameter-dominate-kernel:irreducible}.
\end{proof}

Note that in the important case of $C_0$-semigroups of Markov operators on $L^1$-spaces, the above theorem also follows from \cite[Thm~1]{pichor2000} although this was not explicitly stated in \cite{pichor2000}.

The assumption that the underlying measure space contains an atom is frequently fulfilled in models from queuing and reliability theory. 
The state spaces in such models are often isometrically lattice isomorphic to an $L^1$-space over a measure space containing an atom, 
and the time evolution of the system is described by means of a $C_0$-semigroup consisting of Markov operators; see e.g.\ \cite{gupur2011a} for an introduction. 
For several of these models one can prove convergence of the semigroup 
with respect to the operator norm as e.g.\ in \cite{zheng2011, wang2012, zheng2015}. 
On the other hand, there  also exist models for which one can only show strong stability; in \cite{haji2007, gupur2011, haji2013} such results are proved by employing 
spectral analysis of the generator and a version of the ABLV-theorem. 
Theorem~\ref{thm:semigroup-on-measure-space-with-atom} allows for a simplification of several of these arguments.

\subsection*{Semigroups on spaces of continuous functions}

We complete our discussion with an analysis of one-parameter semigroups on spaces of continuous functions or, more generally, on AM-spaces.
Let $L$ be a locally compact Hausdorff space and let $C_0(L)$ denote the space of real-valued continuous functions on $L$ that vanish at infinity. 
We would like to apply our results to prove strong convergence of positive semigroups on such spaces. 
Unfortunately though, the norm of the space $C_0(L)$ is in general not order continuous, so we cannot apply Theorem~\ref{thm:convergence-partial-AM-compact} nor any of its consequences from Sections~\ref{sec:consequences-i-convergence-results} and~\ref{sec:consequences-ii-spectral-theoretic-results}.
On the other hand, Theorem~\ref{thm:convergence-AM-compact} and its consequences are not particularly well-suited for this type of spaces, either: 
in the important special case where $L$ is compact, each AM-compact operator is automatically compact and for positive semigroups containing 
a compact operator one even has convergence in operator norm under very weak assumptions, cf.~\cite[Thm~4]{lotz1986}.

Still, we can derive new and non-trivial results for positive semigroups on AM-spaces by duality arguments.

\begin{theorem} \label{thm:continuous-functions-dominate-compact-operator}
	Let $\cT = (T_t)_{t \in (0,\infty)}$ be a bounded, positive and irreducible semigroup on an AM-space $E$.
	If there exists a time $s\in (0,\infty)$ and a compact operator $K \in \calL(E)$ such that $0 < K \le T_s$, 
	then $(1)_{t \in (0,\infty)}$ is the only possible unimodular eigenvalue of $\cT$.
\end{theorem}

In order to prove the theorem, we make use of the following lemma.

\begin{lemma} \label{lem:unimodular-eigenvalue-of-dual-semigroup}
	Let $\cT = (T_t)_{t \in (0,\infty)}$ be a bounded one-parameter semigroup on a complex Banach space $E$. 
	Then any unimodular eigenvalue of $\cT$ is also an eigenvalue of the dual semigroup $\cT' \coloneqq (T_t')_{t \in (0,\infty)} \subseteq \calL(E')$.
\end{lemma}
\begin{proof}
	By passing to the semigroup $((\lambda_t)^{-1}T_t)_{t \in (0,\infty)}$ we may assume that $\lambda_t = 1$ for all $t \in (0,1)$. Now the assertion follows from Day's fixed pointed theorem \cite[Theorem~3]{day1961}, just as in the proof of Proposition~\ref{prop:conditions-in-convergence-theorem}(c).
\end{proof}

\begin{proof}[Proof of Theorem~\ref{thm:continuous-functions-dominate-compact-operator}]
	Let $(\lambda_t)_{t \in (0,\infty)}$ be a unimodular eigenvalue of $\cT$. 
	Then it follows from Lemma~\ref{lem:unimodular-eigenvalue-of-dual-semigroup} that $(\lambda_t)_{t \in (0,\infty)}$ is 
	also an eigenvalue of the dual semigroup $\cT' \coloneqq (T_t')_{t \in (0,\infty)} \subseteq \calL(E')$. 
	We check that $\cT'$ satisfies all assumptions of Theorem~\ref{thm:eigenvalues-partial-AM-compact}.
	
	Since $E'$ is an AL-space, its norm is  order continuous; see Corollary~2.4.13 and the discussion below Definition~2.4.11 in~\cite{meyer1991}. 
	To see that every super fixed point of $\cT'$ is a fixed point, one argues as follows. 
	Let $z \in E_\C \setminus \{0\}$ be an eigenvector of $\cT$ for the eigenvalue $(\lambda_t)_{t \in (0,\infty)}$,
	where $E_\C$ denotes the Banach lattice complexification of $E$. 
	Then $\abs{z}$ is a super fixed point of $\cT$.
	Since $\cT$ is irreducible, $\abs{z}$ is a fixed point of $\cT$ by Proposition~\ref{prop:conditions-in-convergence-theorem}(c)
	and a quasi-interior point of $E_+$. 
	Hence, we have $\applied{\varphi}{\abs{z}} > 0$ for all $\varphi \in E'_+ \setminus \{0\}$, so $\abs{z} \in E \subseteq E''$ acts as a strictly positive functional on $E'$.
	Since $\abs{z}$ is clearly a fixed point of $(T_t'')_{t \in (0,\infty)} \subseteq \calL(E'')$, it follows from Proposition~\ref{prop:conditions-in-convergence-theorem}(b) 
	that every super fixed point of $\cT'$ is indeed a fixed point.
	
	It remains to show that assumption~(b) of Theorem~\ref{thm:eigenvalues-partial-AM-compact} is fulfilled for $\cT'$, so let $0 < \varphi \in E'$ be a fixed point of $\cT'$. 
	Then $\varphi$ is strictly positive since $\cT$ is irreducible. 
	The operator $K' \in \calL(E')$ is compact, thus AM-compact, and fulfils $0 \le K' \le T_s'$. 
	Since $K > 0$, there exists $0 < f \in E$ such that $Kf > 0$ and hence $\applied{K'\varphi}{f} = \applied{\varphi}{Kf} > 0$. Therefore, $K' \varphi > 0$
	and all assumptions of Theorem~\ref{thm:eigenvalues-partial-AM-compact} are fulfilled. This implies that $\lambda_t = 1$ for all $t \in (0,\infty)$.
\end{proof}

Note that, even in case that $\cT$ is a $C_0$-semigroup, 
the preceding proof requires results for semigroups with less time regularity since the dual semigroup $\cT'$ is in general not strongly continuous. 
The proof also demonstrates another advantage of our approach:
since the dual space of an AM-space is, in general, quite large we cannot expect the dual semigroup $\cT'$ to be irreducible, even though the semigroup $\cT$ itself is. 
Hence, the above proof only works because Theorem~\ref{thm:eigenvalues-partial-AM-compact} (respectively, the underlying Theorem~\ref{thm:convergence-partial-AM-compact}) 
is formulated in a very general version which does not require irreducibility.

Let us close the article with the following consequence of Theorem~\ref{thm:continuous-functions-dominate-compact-operator}. 
If a locally compact Hausdorff space $L$ contains an isolated point and if $E = C_0(L)$, then the assumption in Theorem~\ref{thm:continuous-functions-dominate-compact-operator} 
that $T_s$ dominate a non-zero compact operator is automatically fulfilled and we obtain the following corollary.

\begin{corollary} \label{cor:continuous-functions-isolated}
	Let $L$ be a locally compact Hausdorff space which contains an isolated point and let $\cT = (T_t)_{t \in (0,\infty)}$ be a bounded, positive and irreducible 
	semigroup on $E \coloneqq C_0(L)$. 
	Then $(1)_{t \in (0,\infty)}$ is the only possible unimodular eigenvalue of $\cT$.
\end{corollary}

\appendix

\section{Kernel operators and AM-Compactness} \label{appendix:AM-compact}

In this appendix we recall a few facts about so-called kernel operators on Banach lattices and their relation to integral operators an $L^p$-spaces.

\begin{definition}
	\label{def:kerneloperator}
	Let $E$ and $F$ be Banach lattices such that $F$ is order complete. 
	We denote by $E'\otimes F$ the space of all finite rank operators from $E$ to $F$.
	The elements of $(E'\otimes F)^{\bot\bot}$, the band generated by $E'\otimes F$ in $\cL^r(E,F)$,
	are called \emph{kernel operators}.
\end{definition}

While this definition is rather abstract, there exists a concrete description of kernel operators on concrete function spaces.
In fact, each positive kernel operator on an $L^p$-space can be represented as an \emph{integral operator}. 
Let us first make precise what we mean by this notion:

\begin{definition}
	\label{def:integral-operator}
	Let $p,q \in [1,\infty)$, let $(\Omega_1,\mu_1)$ and $(\Omega_2,\mu_2)$ be $\sigma$-finite measure spaces
	and let $(\Omega,\mu)$ denote their product space. A bounded linear operator
	\[ T\colon L^p(\Omega_1,\mu_1)\to L^q(\Omega_2,\mu_2) \]
	is called an \emph{integral operator} if there exists a measurable function $k \colon  \Omega \to \R$ 
	such that for every $f\in L^p(\Omega_1,\mu_1)$ and for $\mu_2$-almost every $y\in \Omega_2$ 
	the function $f(\argument)k(\argument,y)$ is contained in $L^1(\Omega_1,\mu_1)$ and the equality
	\begin{align}
		\label{eqn:integralop}
		Tf = \int_{\Omega_1} f(x)k(x,\argument) \dx \mu_1(x).
	\end{align}
	holds in $L^q(\Omega_2,\mu_2)$. In this case the function $k$ (which is uniquely determined up to a nullset) 
	is called the \emph{integral kernel} of $T$.
\end{definition}

The relation between kernel operators and integral operators is described by the following proposition. 
It shows that kernel operators can be viewed as an abstract analogue of integral operators.

\begin{proposition}
	\label{prop:concrete-abstract-kernels}
	Let $p,q \in [1,\infty)$, let $(\Omega_1,\mu_1)$ and $(\Omega_2,\mu_2)$ be $\sigma$-finite measure spaces. A linear operator
	\[ T\colon L^p(\Omega_1,\mu_1)\to L^q(\Omega_2,\mu_2) \]
	is a positive kernel operator if and only if it is an integral operator whose integral kernel is positive almost everywhere on $\Omega_1 \times \Omega_2$.
\end{proposition}
\begin{proof}
	See \cite[Prop IV 9.8]{schaefer1974}.
\end{proof}

In \cite[Prop IV 9.8]{schaefer1974} the reader can also find a similar result about non-positive 
kernel operators and information about the case where $p$ or $q$ equals $\infty$. For more information
about kernel operators and integral operators we refer to \cite{schep1977, vietsch1979}.

The following proposition explains why our main results from Section~\ref{sec:convergence} are applicable to semigroups containing a kernel operator.

\begin{proposition}
	\label{prop:kernel-AM-compact}
	Let $E$ and $F$ be Banach lattices where the norm on $F$ is order continuous.
	Then every kernel operator $T\in (E\otimes F)^{\bot\bot}$ is AM-compact.
\end{proposition}
\begin{proof}
	See \cite[Cor 3.7.3]{meyer1991}.
\end{proof}

Propositions~\ref{prop:concrete-abstract-kernels} and~\ref{prop:kernel-AM-compact} together show 
that every positive integral operator on $L^p$ is AM-compact. 
For the convenience of the reader we include a more direct proof of this result in the following proposition.

\begin{proposition}
	\label{prop:integraloperatorsAMcompact}
	Let $(\Omega_1,\mu_1)$ and $(\Omega_2,\mu_2)$ be $\sigma$-finite measure spaces
	and $p,q \in [1,\infty)$.
	Let $T\colon L^p(\Omega_1,\mu_1) \to L^q(\Omega_2,\mu_2)$ be a positive 
	integral operator with integral kernel $k \colon \Omega_1\times \Omega_2 \to \R$ according to Definition~\ref{def:integral-operator}.
	Then $T$ is AM-compact.
\end{proposition}
\begin{proof}
	First note that since $T$ is positive, the kernel function $k$ is positive almost everywhere by \cite[Prop 3.3.1]{meyer1991}.
	Let $0\leq f \in L^p(\Omega_1,\mu_1)$ and $(f_n)\subseteq [0,f]$.
	After passing to a subsequence, we may assume that $(f_n)$ converges weakly to some $g\in [0,f]$ by the order continuity of the $L^p$-norm.
	As $\Omega_1\times\Omega_2$ is $\sigma$-finite, we find a monotonically increasing
	sequence $(k_m)$ of positive and simple functions on $\Omega_1\times\Omega_2$ such that $(\mu_1\otimes\mu_2)(\{k_m > 0\}) <\infty$ 
	and  $\lim k_m(x,y) = k(x,y)$ almost everywhere. Since, in addition, each $k_m$ is bounded, we have that
	$k_m(\argument,y) \in L^{p'}(\Omega_1,\mu_1)$ for each $m\in\N$ and almost every $y\in \Omega_2$, where $\frac{1}{p}+\frac{1}{p'} = 1$. Hence, 
	\begin{align}
	\label{eqn:fntog}
	\lim_{n\to\infty} \int_{\Omega_1} f_n(x)k_m(x,y) \dx\mu_1(x) = \int_{\Omega_1} g(x) k_m(x,y) \dx\mu_1(x) 
	\end{align}
	for almost every $y\in \Omega_2$ by the definition of $g$.

	Now fix $y\in \Omega_2$ such that $k_m(\argument,y) \in L^{p'}(\Omega_1,\mu_1)$, 
	$\lim k_m(x,y) = k(x,y)$ for $\mu_1$-almost every $x\in \Omega_1$ and such that 
	$(Tg)(y)$, $(Tf)(y)$ and all $(Tf_n)(y)$ are given according to \eqref{eqn:integralop}. Let $\eps>0$.
	By dominated convergence we find $m\in\N$ such that 
	\[ 0\leq \int_{\Omega_1} f(x)\big(k(x,y) - k_m(x,y)\big) \dx\mu_1(x) \leq \eps \]
	and hence, since $g,f_n \in [0,f]$, this implies that
	\[  0\leq \int_{\Omega_1} g(x) \big(k(x,y)-k_m(x,y)\big) \dx\mu_1(x)\leq \eps \]
	and
	\[  0\leq \int_{\Omega_1} f_n(x) \big(k(x,y)-k_m(x,y)\big) \dx\mu_1(x)\leq \eps \]
	for all $n\in\N$. Now pick $n_0\in \N$ such that $\abs{\applied{f_n-g}{k_m(\argument,y)}}< \eps$ for all $n\geq n_0$.
	Then we obtain that
	\begin{align*}
		&\abs{(Tf_n)(y) - (Tg)(y)} = \abs*{\int_{\Omega_1} \big(f_n(x)-g(x)\big) k(x,y) \dx\mu_1(x) } \\
		&\quad = \abs[\Bigg]{\int_{\Omega_1} f_n(x) \big(k(x,y)-k_m(x,y)\big)\dx\mu_1(x) - \int_{\Omega_1} \big(g(x)-f_n(x)\big) k_m(x,y) \dx\mu_1(x)  \\
		&\qquad - \int_{\Omega_1} g(x) \big(k(x,y)-k_m(x,y)\big)\dx\mu_1(x)}\\
		&\quad \leq \int_{\Omega_1} f_n(x) \big(k(x,y)-k_m(x,y)\big)\dx\mu_1(x) + \abs*{\int_{\Omega_1} \big(g(x)-f_n(x)\big) k_m(x,y) \dx\mu_1(x) } \\
		&\qquad + \int_{\Omega_1} g(x) \big(k(x,y)-k_m(x,y)\big)\dx\mu_1(x) \leq 3\eps
	\end{align*}
	for all $n\geq n_0$. Since this holds for almost all $y\in \Omega_2$, we 
	actually proved that $(Tf_n)$ converges to $Tg$ pointwise almost everywhere.
	As $(Tf_n)\subseteq [0,Tf]$ it follows from the dominated convergence theorem that $(Tf_n)$ converges to $Tg$ in
	$L^q$-norm. This readily shows that $T[0,f]$ is compact. Since $f$ was arbitrary, $T$ is AM-compact.
\end{proof}

\section{A few facts from group theory} \label{appendix:group-theory}

In this appendix we recall a few facts from group theory which are used in the main text. As said in the preliminaries, for a group $G$ and $A \subseteq G$ we denote by $\langle A \rangle$ the smallest subgroup of $G$ containing $A$. 
Note that if all elements of $A$ commute, then $\langle A \rangle$ is commutative.
This follows from that fact that if $a \in G$ commutes with $b \in G$, then $a$ commutes with $b^{-1}$ and $a^{-1}$ commutes with $b$. We adopt to the convention of denoting the binary operation in commutative groups by `$+$'.

For a normal subgroup $N$ of a group $G$ the \emph{index} of $N$ in $G$ is defined to be the cardinality of the quotient group $G/N$. We call a subgroup of $G$ proper if it does not coincide with $G$.

\begin{proposition} \label{prop:divisible-groups}
	For a group $G$ consider the following statements:
	\begin{enumerate}[\upshape (i)]
		\item The group $G$ is divisible.
		\item Every proper normal subgroup of $G$ has infinite index.
		\item Every group homomorphism from $G$ to any finite group is trivial.
		\item Every group action of $G$ on a finite set is trivial.
	\end{enumerate}
	Then {\upshape(i) $\Rightarrow$ (ii) $\Leftrightarrow$ (iii) $\Leftrightarrow$ (iv)}. If, in addition, $G$ is commutative, then also {\upshape(ii) $\Rightarrow$ (i)}.
\end{proposition}
\begin{proof}
	The proofs of the equivalences (ii) $\Leftrightarrow$ (iii) $\Leftrightarrow$ (iv) are straightforward, so we omit them. If~(i) holds and $H \subseteq G$ is a normal subgroup of index $n \in \N$, then we can find for each $t \in G$ an element $s \in G$ such that $s^n = t$; this yields $tH = s^nH = H$, so $H = G$.
	
	It remains to show (iii) $\Rightarrow$ (i) in case $G=(G,+)$ is commutative. If $G$ is not divisible, 
	we can find a prime number $p \in \N$ and an element of $G$ which cannot be divided by $p$. 
	Therefore, the subgroup $H \coloneqq \{p t:  t \in G\}$ of $G$ is proper.
	
	Since the order of every element of $G/H$ is a divisor of $p$,
	every non-trivial element of $G/H$ has order $p$.  Hence, the mapping
	\begin{align*}
		\Z/p\Z \times G/H \to G/H, \quad (n + p\Z, t+H) \mapsto nt+H.
	\end{align*}
	is well-defined and, with respect to this scalar multiplication and the addition on $G/H$, $G/H$ is a vector space of dimension at least $1$ over the field $\Z/ p \Z$.
	
	We choose a subspace $V$ of $G/H$ of co-dimension $1$. Then the factor group $(G/H)/V$ is isomorphic to $\Z/p \Z$, so there exists a surjective group homomorphism $G \to G/H \to \Z / p \Z$, meaning that (iii) fails.
\end{proof}

For the purpose of reference in the main text we also state the following proposition explicitly, although we omit its simple proof.

\begin{proposition} \label{prop:generating-semigroup}
	Let $G, H$ be groups and assume that $G$ is commutative and generated by a subsemigroup $S \subseteq G$. Then every semigroup homomorphism $\varphi: S \to H$ can be extended in a unique way to a group homomorphism $\psi\colon G \to H$.
\end{proposition}

\bibliographystyle{abbrv}
\bibliography{literature}

\end{document}